\documentclass[12pt,verbatim]{amsart}
\usepackage{amsmath,amssymb,amsfonts}  
\usepackage{graphicx}
\usepackage{caption}
\usepackage{amsthm}
\usepackage{subfigure}
\usepackage{verbatim}
\usepackage{amscd}
\usepackage{ifthen}

\newtheorem{theorem}[equation]{Theorem}

\newtheorem{lemma}[equation]{Lemma}
\newtheorem{proposition}[equation]{Proposition}

\newtheorem{remark}[equation]{Remark}

\newtheorem{nonsec}[equation]{}

\newcommand{\R}{{\mathbb R}}

\newcommand{\Bn}{ {\mathbb{B}^n} }
\newcommand{\Rn}{ {\mathbb{R}^n} }

\newcommand{\sh}{\,\textnormal{sh}}
\newcommand{\ch}{\,\textnormal{ch}}

\numberwithin{equation}{section}

\title{Remarks on the scale invariant Cassinian metric}
\author{Gendi Wang} 
\address{School of Science, Zhejiang Sci-Tech University, Hangzhou 310018, China}  
\email{gendi.wang@zstu.edu.cn} 

\author{Xiaoxue Xu} 
\address{School of Science, Zhejiang Sci-Tech University, Hangzhou 310018, China}
\email{xiaoxue\_xu@126.com}

\author{Matti Vuorinen} 
\address{Department of Mathematics and Statistics, University of Turku, Turku 20014, Finland}
\email{vuorinen@utu.fi}

\begin{document}  

\newcounter{minutes}\setcounter{minutes}{\time}
\divide\time by 60
\newcounter{hours}\setcounter{hours}{\time}
\multiply\time by 60 \addtocounter{minutes}{-\time}
\def\thefootnote{}
\footnotetext{ {\tiny File:~\jobname.tex,
          printed: \number\year-\number\month-\number\day,
          \thehours.\ifnum\theminutes<10{0}\fi\theminutes }}
\makeatletter\def\thefootnote{\@arabic\c@footnote}\makeatother

\maketitle

\begin{abstract}
We study the geometry of the scale invariant Cassinian metric
and prove sharp comparison inequalities between this  metric and the hyperbolic metric
in the case when the domain is either the unit ball or the upper half space.
We also prove  sharp distortion inequalities for the scale invariant Cassinian metric under M\"obius transformations.
\end{abstract}

{\small \sc Keywords.} {the scale invariant Cassinian metric, the hyperbolic metric, M\"obius transformation }

{\small \sc 2010 Mathematics Subject Classification.} {30F45 (51M10)}

\section{Introduction}

In the Euclidean space $\Rn, n\ge 2\,,$
the natural way to measure distance between two points $x,y \in \Rn$
is to use the length $|x-y|$ of the segment joining the points.
In geometric function theory \cite{gh},
one studies functions defined in subdomains $D \subset \Rn\,,$ and measures distances between two points $x,y \in D\,.$
In this case the Euclidean distance is no longer an adequate method for measuring the distance,
because one has to take into account also the position of the points relative to the boundary $\partial D \,.$

During the past few decades,
many authors have suggested metrics for this purpose.
In the case of the simplest domain, the unit ball $\Bn\,,$
we have the hyperbolic or Poincar\'e metric that is the most common metric in this case.
Therefore, it is a natural idea to analyze the various equivalent definitions of the hyperbolic metric
and to use these to generalize,  if possible, the hyperbolic metric to the case of a given domain $D \subset \Rn\,.$
These generalizations capture usually some but not all features of the hyperbolic metric and are thus called hyperbolic type metrics \cite{b98,fmv,gp,h0,h,i1,Ibragimov,I2,klvw,s99}.

Because the usefulness of a metric depends on how well its invariance properties match those of the function spaces studied,
we now analyze hyperbolic type metrics from this point of view.
The best we can expect is invariance in the same sense as the hyperbolic metrics are invariant,
namely invariance under M\"obius transformations of the M\"obius space
$( \overline{\mathbb{R}}^n, q)\, ,  \overline{\mathbb{R}}^n =  {\mathbb{R}}^n \cup \{ \infty\}\,,$ equipped with the chordal metric $q \,.$
Another useful notion is invariance with respect to similarity transformations.
A similarity transformation is a transformation of the form $x \mapsto \lambda U(x) + b$ where $\lambda >0\,,$ $b \in \mathbb{R}^n\,,$
and $U$ is an orthogonal map, i.e., a linear map with $|U(x)|=|x|$ for all $x \in \mathbb{R}^n\,.$

The quasihyperbolic and the distance ratio metrics introduced by Gehring and Palka \cite{gp}
have become widely used hyperbolic type metrics in geometric function theory in plane and space \cite{gh}.
Both metrics are defined for subdomains of $ {\mathbb{R}}^n \,$ and are invariant under similarity transformations, but they are not M\"obius invariant.
M\"obius invariant metrics, defined in terms of the absolute ratios of quadruples of points, were studied by several authors
in the case of a general domain $D \subset \overline{\mathbb{R}}^n$ with ${\rm card} (\overline{\mathbb{R}}^n \setminus D  ) \ge 2\,.$
These metrics include the Apollonian metric of Beardon \cite{b98},
the M\"obius invariant metric of Seittenranta  \cite{s99},
and the generalized hyperbolic metric of H\"ast\"o \cite{h}.
Each of these metrics generates its own geometry and the study of transformation rules of these metrics
under M\"obius transformations and conformal mappings are natural questions to study.
If we can describe the balls of a metric space ''explicitly'',
then we already know a lot about the geometry of the metric --
this requires that we can estimate the metric in terms of well-known metrics.
For a survey and comparison inequalities between some of these metrics, see \cite{dhv,hkvz,h1,MS,MS2,s99,wv,z}.

Recently Ibragimov \cite{Ibragimov} introduced the scale invariant Cassinian metric $\tilde{\tau}_{D}$, see the definition below in \ref{def}.
It is readily seen that the metric $\tilde{\tau}_{D}$
is invariant under similarity transformations.
Several authors \cite{Ibragimov,MS,MS2} have studied some basic properties of the scale invariant Cassinian metric
and its distortion under M\"obius transformations of the unit ball,
and also quasi-invariance properties under quasiconformal mappings.

In this paper,
we will continue this research and study the geometry of the scale invariant Cassinian metric
and establish sharp comparison results between this metric and the hyperbolic metric of the unit ball or of the upper half space,
and also prove sharp distortion inequalities under M\"obius transformations.

\bigskip

\section{Preliminaries}\label{section 2}

\begin{nonsec}{\bf Hyperbolic metric.}
{\rm
The hyperbolic metric $\rho_{\mathbb{B}^n}$ and $\rho_{\mathbb{H}^n}$
of the unit ball ${\mathbb{B}^n}= \{ z\in {\mathbb{R}^n}: |z|<1 \} $
and
of the upper half space ${\mathbb{H}^n} = \{ (x_1,\ldots,x_n)\in {\mathbb{R}^n}:  x_n>0 \}$
can be defined as follows.
By \cite[p.40]{b} we have for $x,y\in\mathbb{B}^n$\,,
\begin{equation}\label{sro}
\sh{\frac{\rho_{\mathbb{B}^n}(x,y)}{2}}=\frac{|x-y|}{\sqrt{1-|x|^2}{\sqrt{1-|y|^2}}}\,,
\end{equation}
and
by \cite[p.35]{b} for  $ x,y\in \mathbb{H}^n$\,,
\begin{equation}\label{cro}
\ch{\rho_{\mathbb{H}^n}(x,y)}=1+\frac{|x-y|^2}{2x_ny_n}\,.
\end{equation}
Two special formulas of the hyperbolic metric are frequently used \cite[(2.17),(2.6)]{vu2}\,:
\begin{equation}\label{bro}
\rho_{\mathbb{B}^n}(r e_1,s e_1)=\log\left(\frac{1+s}{1-s}\cdot\frac{1-r}{1+r}\right)\,,\quad {\rm for }\quad -1<r<s<1\,,\quad s>0\,,
\end{equation}
and
\begin{equation}\label{hro}
\rho_{\mathbb{H}^n}(r e_n,s e_n)=\log \frac sr\,,\quad {\rm for }\quad 0<r<s\,.
\end{equation}
}
\end{nonsec}
\begin{nonsec}{\bf Absolute ratio.}
{\rm
For a quadruple of distinct points  $a,b,c,d, \in \overline{\mathbb{R}}^n$\,,
the absolute ratio is defined as
$$
|a,b,c,d|= \frac{q(a,c)q(b,d)}{q(a,b) q(c,d)}\,,
$$
where $q(a,c)$ is the chordal distance \cite[(1.14)]{vu2}.
The most important property of the absolute ratio is its invariance under M\"obius transformations \cite[Theorem 3.2.7]{b}.
For the basic properties of M\"obius transformations the reader is referred to \cite{b}\,.
}
\end{nonsec}

In terms of the absolute ratio,
the hyperbolic metric $\rho_D\,,$ $D \in \{\mathbb{B}^n,\mathbb{H}^n\}\,,$ can be defined for $x,y \in D$ as follows \cite[(7.26)]{b}\,:
\begin{equation}\label{absratrho}
\rho_D(x,y) = \sup \{  \log|u,x,y,v|:   u,v \in \partial D\} \,.
\end{equation}

Because of the M\"obius invariance of the absolute ratio and \eqref{absratrho}\,,
we may define for every M\"obius transformation $g$ the hyperbolic metric in $g(\mathbb{B}^n)$.
This metric will be denoted by $\rho_{g(\mathbb{B}^n)}$.
In particular, if $g: {\mathbb{B}^n}\to {\mathbb{H}^n}$ is a M\"obius transformation with $g({\mathbb{B}^n})={\mathbb{H}^n},$
then for all $x,y \in{\mathbb{B}^n}$ there holds $\rho_{\mathbb{B}^n}(x,y)= \rho_{\mathbb{H}^n}(g(x),g(y))\,.$

\begin{nonsec}{\bf Distance ratio metric.}\label{jmetric}
{\rm
For a proper open subset $D$ of $\mathbb{R}^{n}$ and for $x,y\in D$,
the  distance ratio metric $j_D$ is defined as
$$
 j_D(x,y)=\log \left( 1+\frac{|x-y|}{\min \{d(x, \partial D),d(y, \partial D) \} } \right)\,,
$$
where $d(x, \partial D)$ denotes the Euclidean distance from the point $x$ to the boundary $\partial D$.
The distance ratio metric was introduced by Gehring and Palka \cite{gp}
and in the above simplified form by Vuorinen \cite[(2.34)]{vu2}.
}
\end{nonsec}

The well-known relation between the distance ratio metric and the hyperbolic metric is shown in the following lemma.

\begin{lemma}\label{c} \cite[Lemma 7.56]{avv}\label{le8}, \cite[Lemma 2.41(2)]{vu2}
Let $D\in\{\mathbb{B}^{n},\mathbb{H}^{n}\}$. Then for  all $x,y \in D$, we have
\begin{equation*}
\frac12 \rho_{D}(x,y)\le j_{D}(x,y)\leq\rho_{D}(x,y)\,.
\end{equation*}

\end{lemma}

\medskip

\begin{nonsec}{\bf Cassinian oval.}  
{\rm
A Cassinian oval $C(\alpha_1, \alpha_2; b)$ is defined as
\begin{equation*}
\{z\in \mathbb{R}^2:|z-\alpha_1|\cdot|z-\alpha_2|=b^{2}\}\,,
\end{equation*}
where $b>0$ and $\alpha_1$, $\alpha_2$ are two fixed points called the foci of the oval.

Let  $\alpha_1=(-a,0)$ and $\alpha_2=(a,0)$ with $a\in\R$\,,
then the equation of the Cassinian oval $C(\alpha_1, \alpha_2; b)$ is
\begin{equation}\label{ce}
\left((z_1-a)^{2}+z_2^{2}\right)\left((z_1+a)^{2}+z_2^{2}\right)=b^{4}\,.
\end{equation}
}
\end{nonsec}

\begin{figure}[h]
\centering
\includegraphics[width=9cm]{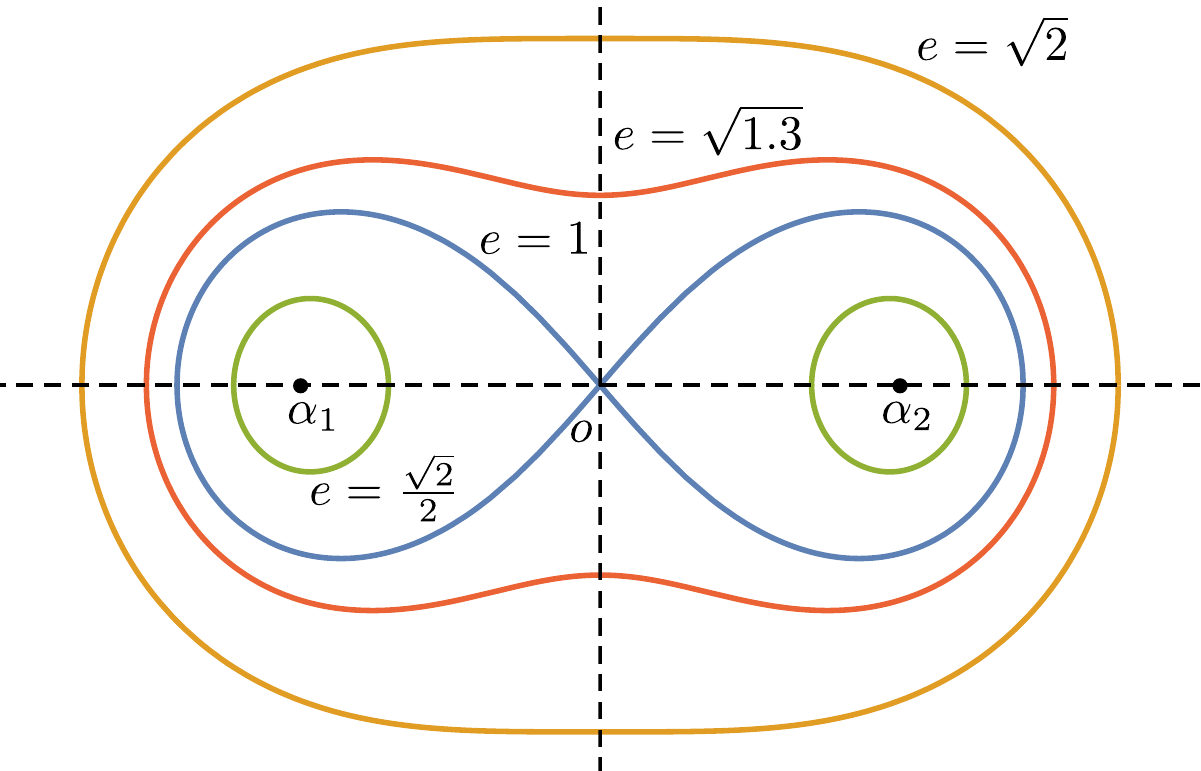}
\caption{\label{tu12}  Cassinian ovals  with $e=\frac{\sqrt{2}}{2}$, $e=1$,  $e=\sqrt{1.3}$, and $e=\sqrt{2}$, respectively.}
\end{figure}

The shape of a Cassinian oval  depends on $e=\frac{b}{a}$ (see Fig.\ref{tu12}).
When $e<1$, the oval consists of two separate loops\,.
When $e=1$, the oval is the lemniscate of Bernoulli having the shape of number eight\,.
When $e>1$, the oval is a single loop enclosing both foci.
Moreover, it is peanut-shaped for $1<e<\sqrt{2}$ and convex for $e\geq\sqrt{2}$.
In the limiting case  $a\rightarrow 0$ the Cassinian oval reduces to  a circle.

\begin{proposition}\label{prop1}
Let $p=(p_1,p_2)$ be a point on the  Cassinian oval $C(\alpha_1, \alpha_2; b)$.
Then the distance from the origin to the point $p$ is increasing as a function of $p_1>0$.
\end{proposition}
\begin{proof}
By \eqref{ce}, we have
\begin{equation*}
|p|^{2}=p_1^{2}+p_2^{2}
=\sqrt{4a^{2}p_1^{2}+b^{4}}-a^{2},
\end{equation*}
which implies that the distance from the origin to the point $p$ is increasing for $p_1>0$\,.
\end{proof}

\medskip

\begin{proposition}\label{prop2}
The Cassinian oval $C(\alpha_1, \alpha_2; b)$ inscribes the circle $\partial{\mathbb{B}^{2}}(\sqrt{a^{2}+b^{2}})$\,.
\end{proposition}
\begin{proof}
By Proposition \ref{prop1} and \eqref{ce}, it is clear that
the maximum distance from the origin to the point $p$ on the Cassinian oval $C(\alpha_1, \alpha_2; b)$ is $\sqrt{a^{2}+b^{2}}$\,.
\end{proof}

\medskip

\begin{nonsec}{\bf Scale invariant Cassinian metric.}\label{def}
{\rm
For a proper subdomain $D$ of $\mathbb{R}^{n}$ and for all $x,y\in D$,
the scale invariant Cassinian metric $\tilde{\tau}_{D}$ is defined as \cite {Ibragimov}
\begin{equation*}
\tilde{\tau}_{D}(x,y)=\log\left(1+\sup_{p\in \partial D}\frac{|x-y|}{\sqrt{|x-p||y-p|}}\right).
\end{equation*}
Geometrically, $\tilde{\tau}_{D}(x,y)$ can be defined by means of the maximal Cassinian oval $C\subset \overline D$ with foci $x,y\in D$.
Then for every point $p\in C$, we have
\begin{equation*}
\tilde{\tau}_{D}(x,y)=\log\left(1+\frac{|x-y|}{\sqrt{|x-p||y-p|}}\right).
\end{equation*}
Because of this geometric interpretation, the  metric $\tilde{\tau}_{D}$ is monotonic with respect to domains, i.e.,
if $D\subset D'$, then $\tilde{\tau}_{D'}(x,y)\leq\tilde{\tau}_{D}(x,y)$ for $x\,,y\in D$.
}
\end{nonsec}

The following lemma shows the relation between the scale invariant Cassinian metric and the distance ratio metric.

\begin{lemma}\cite[Theorem 3.3]{Ibragimov}\label{le7}
Let $D\subset\mathbb{R}^{n}$ be a domain with $\partial D\neq\emptyset$.
For all $x,y\in D$, we have
\begin{equation*}
\frac{1}{2}j_{D}(x,y)\leq\tilde{\tau}_{D}(x,y)\leq j_{D}(x,y)
\end{equation*}
and
\begin{equation*}
\tilde{\tau}_{D}(x,y)\leq \frac{1}{2}j_D(x,y)+\frac{1}{2}\log3\,.
\end{equation*}
All the inequalities are sharp.
\end{lemma}

\medskip

\begin{nonsec}{\bf M\"obius invariant Cassinian metric.}
{\rm
Let $D$ be a subdomain of $\overline{\mathbb{R}}^n$ with ${\rm card} (\partial D)\geq 2$.
For $x,y\in D$, the M\"obius invariant Cassinian metric $\tau_D$ is defined as \cite{I2}
\begin{equation*}
\tau_{D}(x,y)=\log\left(1+\sup_{p,q\in \partial D}\frac{|x-y||p-q|}{\sqrt{|x-p||y-q||x-q||y-p|}}\right).
\end{equation*}
}
\end{nonsec}

Since $\tau_D$ can be expressed as
\begin{equation*}
\tau_{D}(x,y)=\log\left(1+\sup_{p,q\in \partial D}\sqrt{\frac{|x-y||p-q|}{|x-p||y-q|}}\sqrt{\frac{|x-y||q-p|}{|x-q||y-p|}}\right)\,,
\end{equation*}
the M\"obius invariant Cassinian metric is M\"obius invariant. Namely,

\begin{lemma}\cite[Corollary 2.1]{I2}\label{lelip1}
Let $f$ be a M\"{o}bius transformation of $\overline{\mathbb{R}}^{n}$  and $D\subset\overline{\mathbb{R}}^{n}$ with ${\rm card} (\partial D)\geq2$.
Then for all $x,y\in D$, we have
\begin{align*}
\tau_{f(D)}\left(f(x),f(y)\right)=\tau_{D}(x,y)\,.
\end{align*}
\end{lemma}

The following lemma shows the relation between the scale invariant Cassinian metric and the M\"obius invariant Cassinian metric.
\begin{lemma}\cite[Theorem 3.3]{I2}\label{lelip2}
For all $x,y\in D\subsetneq \mathbb{R}^{n}$, we have
\begin{equation*}
 \frac 12\tau_{D}(x,y)\leq \tilde{\tau}_{D}(x,y)\leq \tau_{D}(x,y)\,.
\end{equation*}
\end{lemma}

\bigskip

\section{The estimate of $\tilde{\tau}$-metric}

In this section,
we give the estimate for the scale invariant Cassinian metric in the unit disk or the upper half plane
by studying the formulas of special cases and the geometry  of the $\tilde{\tau}$-metric.
The results can be applied to higher-dimensional cases, e.g.,
the special formulas are used in the proof of the results in Sections \ref{tauro} and \ref{taumob}.
For the convenience, we identify $\R^2$ with the complex plane $\mathbb{C}$ and use complex number notation also if needed in the sequel.

\begin{nonsec}{\bf The unit disk case.}
{\rm
We first study the formulas of special cases of the scale invariant Cassinian metric in the unit disk.
}
\end{nonsec}

\begin{lemma}\label{lem1}
Let $x,y\in\mathbb{B}^{2}\setminus \{0\}$ with $|x|=|y|$.

(1) If $|x+y|\leq \frac{4|x|^2}{1+|x|^2}$, then
\begin{equation*}
\tilde{\tau}_{\mathbb{B}^2}(x,y)=\log\left( 1+\sqrt{\frac{2|x||x-y|}{1-|x|^2}}\right).
\end{equation*}

(2) If $|x+y|>\frac{4|x|^2}{1+|x|^2}$, then
\begin{equation*}
\tilde{\tau}_{\mathbb{B}^2}(x,y)
=\log\left( 1+\frac{|x-y|}{\sqrt{1+|x|^2-|x+y|}}\right).
\end{equation*}
\end{lemma}
\begin{proof}
By symmetry, we may assume that $x=(x_1,x_2)$ and $y=(x_1,-x_2)$, where $x_1=\frac{|x+y|}{2},~x_2=\frac{|x-y|}{2}$.
Let $p=(t,\sqrt{1-t^{2}})$ with $x_1\leq t\le1$ and
\begin{align*}
f(t)&=|x-p|^{2}|y-p|^{2}\\
&=\left((x_1^2+x_2^2+1-2x_1 t)-2x_2\sqrt{1-t^2}\right)\left((x_1^2+x_2^2+1-2x_1 t)+2x_2\sqrt{1-t^2}\right)\\
&=4(x_1^{2}+x_2^{2})t^2-4x_1(x_1^{2}+x_2^{2}+1)t+(x_1^{2}+x_2^{2})^2+2x_1^{2}-2x_2^{2}+1\,.
\end{align*}

(1) If $|x+y|\leq \frac{4|x|^2}{1+|x|^2}$, then
\begin{align*}
f_{\min}(t)=f(t_0)
=\frac{x_2^{2}(1-x_1^{2}-x_2^{2})^{2}}{x_1^{2}+x_2^{2}}\,,
\end{align*}
where $t_{0}=\frac{x_1}{2}(1+\frac{1}{x_1^{2}+x_2^{2}})\,.$
Therefore,
\begin{equation*}
\tilde{\tau}_{\mathbb{B}^2}(x,y)
=\log\left(1+\frac{|x-y|}{\sqrt[4]{ f(t_0)}}\right)
=\log\left( 1+\sqrt{\frac{2|x||x-y|}{1-|x|^2}}\right)\,.
\end{equation*}

\medskip

(2) If $|x+y|>\frac{4|x|^2}{1+|x|^2}$, then
$$f_{\min}(t)=f(1)\,.$$
Therefore,
\begin{equation*}
\tilde{\tau}_{\mathbb{B}^2}(x,y)
=\log\left(1+\frac{|x-y|}{\sqrt[4]{ f(1)}}\right)
=\log\left( 1+\frac{|x-y|}{\sqrt{1+|x|^2-|x+y|}}\right).
\end{equation*}

This completes the proof.
\end{proof}

\medskip

\begin{lemma}\label{lem2}\cite[Proposition 3.1]{MS}
Let $x,y\in\mathbb{B}^{2}$ with $x~=~ty$, $ t\in\mathbb{R}\setminus \{0\}$ and $|x|\leq|y|$.

(1) If $t>0$, then
\begin{equation*}\label{fe}
\tilde{\tau}_{\mathbb{B}^2}(x,y)=\log\left(1+\frac{|x-y|}{\sqrt{(1-|x|)(1-|y|)}}\right).
\end{equation*}

(2) If $t<0$, then
\begin{equation*}\label{ee}
\tilde{\tau}_{\mathbb{B}^2}(x,y)=\log\left(1+\frac{|x-y|}{\sqrt{(1+|x|)(1-|y|)}}\right).
\end{equation*}
\end{lemma}

\medskip

\begin{remark}\label{rmk-tau0x}
{\rm
By definition, it is easy to see that
\begin{equation*}
\tilde{\tau}_{\mathbb{B}^2}(0,x)=\log\left(1+\frac{|x|}{\sqrt{1-|x|}}\right).
\end{equation*}
}
\end{remark}

\medskip

\begin{lemma}\label{lethm1}
$Let~ x,y,x',y'\in\mathbb{B}^2$  with
$x'=\frac{x+y}{2}-\frac{|x-y|}{2}\xi$,~$y'=\frac{x+y}{2}+\frac{|x-y|}{2}\xi$,
~$x''=\frac{x+y}{2}-\frac{|x-y|}{2}\zeta$ and $y''=\frac{x+y}{2}+\frac{|x-y|}{2}\zeta$\,,
where
\begin{align*}
\xi=
\left\{
\begin{array}{ll}
\frac{x+y}{|x+y|},
& x+y\neq0\,,\\
e_1,
& x+y=0\,,
\end{array}
\right.
\end{align*}
and
$\zeta=i\,\xi$\,.
Then
\begin{equation*}
\tilde{\tau}_{\mathbb{B}^2}(x'',y'')\leq\tilde{\tau}_{\mathbb{B}^2}(x,y)\leq\tilde{\tau}_{\mathbb{B}^2}(x',y')\,.
\end{equation*}
\end{lemma}
\begin{proof}
Since the result is trivially true for the case $x=-y$,  by symmetry, we may assume that $x\neq -y$ and $0<\arg\frac{y-x}{y+x}<\frac{\pi}{2}$.

Let $b'\in\mathbb{R}$ such that $C(x',y';b')$ is tangent to $\partial\mathbb{B}^{2}$ at $\xi$.
By Proposition \ref{prop2}, there exists a disk $\mathbb{B}^{2}(o',r)\subset \mathbb{B}^{2}$
such that $C(x',y';b')$ inscribes $\overline{\mathbb{B}^{2}}(o',r)$,
where the center $o'=\frac{x+y}{2}$ and the radius $r=\sqrt{b'^{2}+\frac{|x-y|^{2}}{4}}$.
Moreover, $C(x',y';b') \cap\partial \mathbb{B}^{2}(o',r) \cap\partial \mathbb{B}^{2}$ has one and only one point.

With rotation, $C(x,y;b')\subsetneq\mathbb{B}^{2}$ (see Fig.~\ref{tu7}).
Therefore, there exists a positive number $b\,(> b')$ such that $C(x,y;b)$ is tangent to $\partial\mathbb{B}^{2}$\,.
Hence
$$\tilde{\tau}_{\mathbb{B}^2}(x,y)\leq\tilde{\tau}_{\mathbb{B}^2}(x',y')\,.$$

With rotation and by Proposition \ref{prop1}, $C(x'',y'';b)\subsetneq\mathbb{B}^{2}$ (see Fig.~\ref{tu8}).
Therefore, there exists a positive number $b''\,(>b)$ such that $C(x'',y'';b'')$ is tangent to $\partial\mathbb{B}^{2}$\,.
Hence
$$\tilde{\tau}_{\mathbb{B}^2}(x'',y'')\leq\tilde{\tau}_{\mathbb{B}^2}(x,y)\,.$$

This completes the proof.
\end{proof}

\begin{figure}[h]
\begin{minipage}[t]{0.4\linewidth}
\centering
\includegraphics[width=6.2cm]{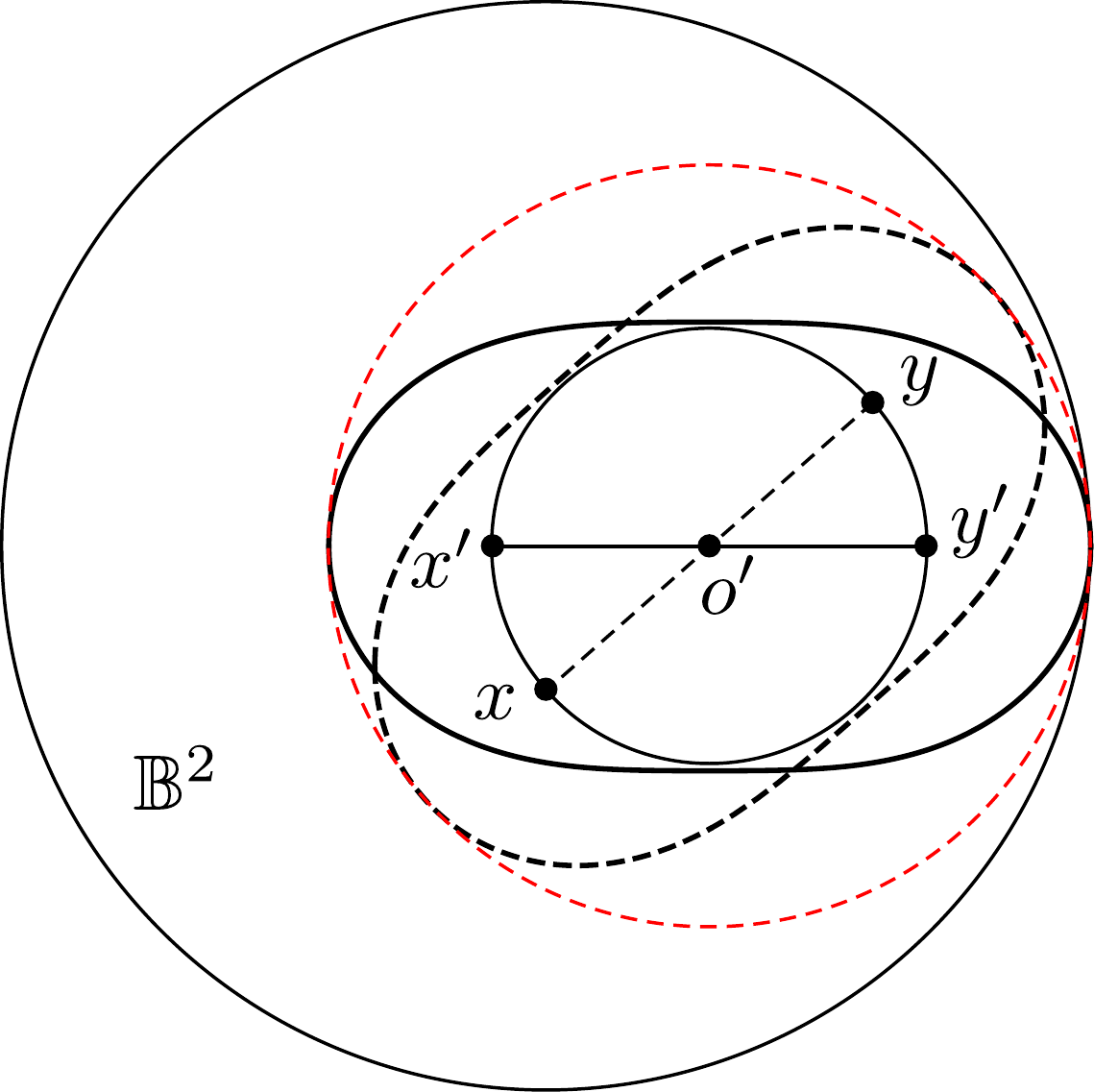}
\caption{\label{tu7} The Cassinian oval $C(x',y'; b')$ is tangent to $\partial\mathbb{B}^{2}$ while $C(x,y; b')\subsetneq\mathbb{B}^{2}$.}
\end{minipage}
\hspace{1.2cm}
\begin{minipage}[t]{0.4\linewidth}
\centering
\includegraphics[width=6.2cm]{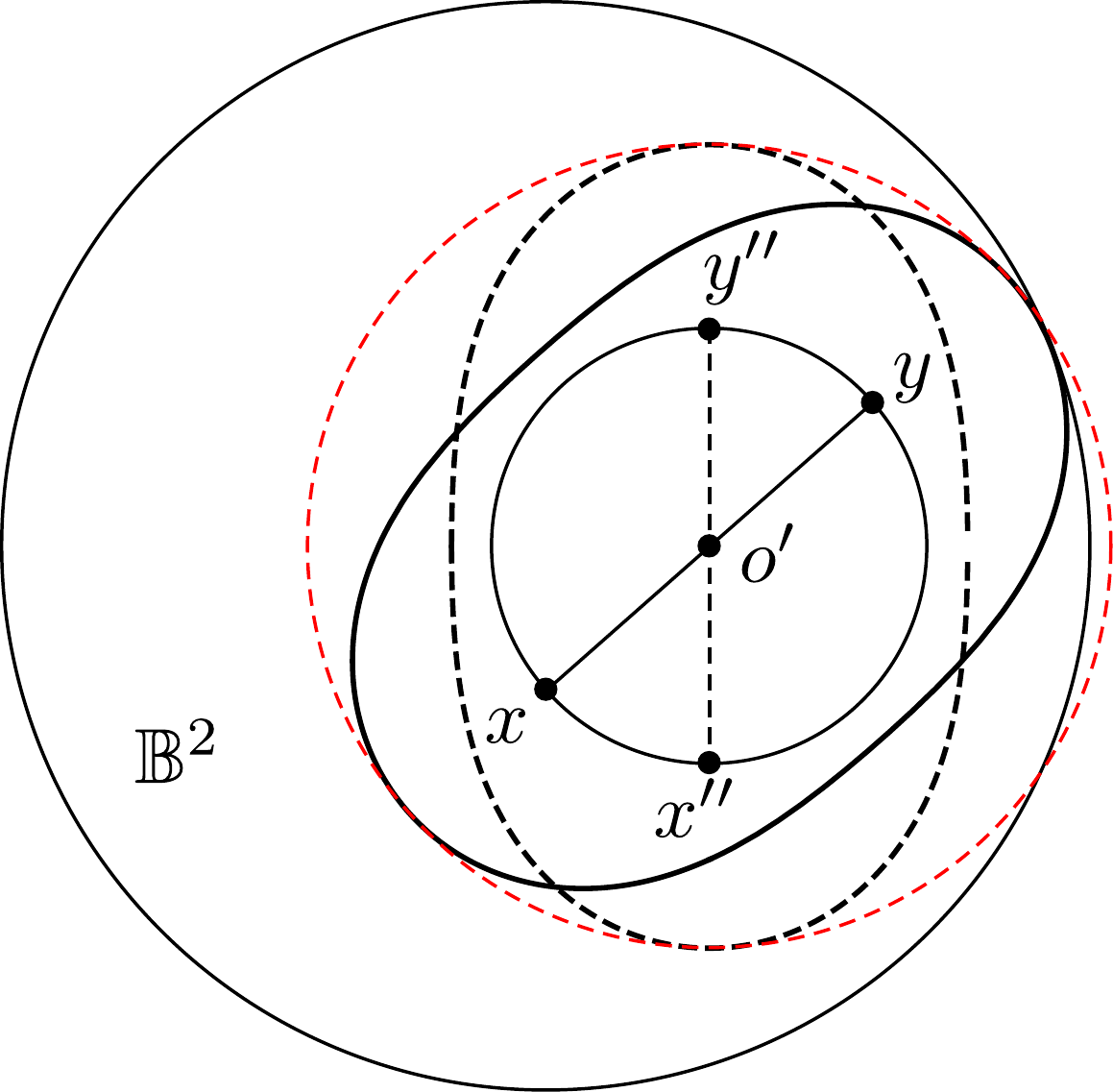}
\caption{\label{tu8} The Cassinian oval $C(x,y; b)$ is tangent to $\partial\mathbb{B}^{2}$ while $C(x'',y''; b)\subsetneq\mathbb{B}^{2}$. }
\end{minipage}
\end{figure}

\medskip

\begin{theorem}\label{co1}
For $x,y\in\mathbb{B}^2$, we have
\begin{align}\label{thm1ineq1}
\tilde{\tau}_{\mathbb{B}^2}(x,y)\geq
\left\{
\begin{array}{ll}
\log\left(1+2\sqrt{\frac{|x-y|\sqrt{|x+y|^{2}+|x-y|^{2}}}{4-|x+y|^{2}-|x-y|^{2}}}\right),
& |x+y|\left(1+\frac{4}{|x+y|^{2}+|x-y|^{2}}\right)\le 4\,,\\
\log\left(1+\frac{2|x-y|}{\sqrt{(2-|x+y|)^{2}+|x-y|^{2}}}\right),
& |x+y|\left(1+\frac{4}{|x+y|^{2}+|x-y|^{2}}\right)>4,
\end{array}
\right.
\end{align}
and
\begin{equation}\label{thm1ineq2}
\tilde{\tau}_{\mathbb{B}^2}(x,y)\leq\log\left(1+\frac{2|x-y|}{\sqrt{(2-|x+y|)^{2}-|x-y|^{2}}}\right),
\quad |x+y|+|x-y|<2\,.
\end{equation}
\end{theorem}
\begin{proof}
Since the equalities in \eqref{thm1ineq1} and \eqref{thm1ineq2} clearly hold when $x=-y$, we may assume that $x\neq-y$ in the sequel.

To prove inequalities \eqref{thm1ineq1}, let $x''\,,y''$ be the same as in  Lemma \ref{lethm1}.
Then
$$x''=\frac{x+y}{2\,|x+y|}\left( |x+y| -i\, |x-y| \right)\quad{\rm and}\quad y''=\frac{x+y}{2\,|x+y|}\left(|x+y| +i\, |x-y|\right)\,.$$

\medskip

{\bf Case 1.} If $|x+y|\left(1+\frac{4}{|x+y|^{2}+|x-y|^{2}}\right)\le 4$, then ${|x''+y''|}\le \frac{4|x''|^2}{1+|x''|^2}$.
By Lemma \ref{lethm1} and Lemma \ref{lem1}(1), we have
\begin{equation*}
\tilde{\tau}_{\mathbb{B}^2}(x,y)\geq\tilde{\tau}_{\mathbb{B}^2}(x'',y'')
=\log\left(1+2\sqrt{\frac{|x-y|\sqrt{|x+y|^{2}+|x-y|^{2}}}{4-|x+y|^{2}-|x-y|^{2}}}\right)\,.
\end{equation*}

\medskip

{\bf Case 2.} If $|x+y|\left(1+\frac{4}{|x+y|^{2}+|x-y|^{2}}\right)> 4$, then ${| x''+y'' |}> \frac{4|x''|^2}{1+|x''|^2}$.
By Lemma \ref{lethm1} and Lemma \ref{lem1}(2), we have
\begin{equation*}
\tilde{\tau}_{\mathbb{B}^2}(x,y)\geq\tilde{\tau}_{\mathbb{B}^2}(x'',y'')
 =\log\left(1+\frac{2|x-y|}{\sqrt{(2-|x+y|)^{2}+|x-y|^{2}}}\right)\,.
\end{equation*}

\medskip

To prove inequality \eqref{thm1ineq2}, let $x'\,,y'$ be the same as in  Lemma \ref{lethm1}.
Then
$$x'=\frac{x+y}{2|x+y|}\left( |x+y|-|x-y|\right)\quad {\rm and}\quad y'=\frac{x+y}{2|x+y|}\left(|x+y|+|x-y|\right)\,.$$
It is easy to see that $x'=ty'$ and $|x'|\leq|y'|$.

\medskip

{\bf Case 3.} If $t=0$, then $|x'|=0$ and hence $|x+y|=|x-y|$.
By Lemma \ref{lethm1} and Remark \ref{rmk-tau0x}, it is clear that inequality \eqref{thm1ineq2} holds.

\medskip

{\bf Case 4.} If $t>0$, then $|x'|=\frac12(|x+y|- |x-y|)$ and $|y'|=\frac12(|x+y|+|x-y|)$\,.
By Lemma \ref{lethm1} and Lemma \ref{lem2}(1), we have
\begin{equation*}
\tilde{\tau}_{\mathbb{B}^2}(x,y)\le\tilde{\tau}_{\mathbb{B}^2}(x',y')
=\log\left(1+\frac{2|x-y|}{\sqrt{(2-|x+y|)^{2}-|x-y|^{2}}}\right).
\end{equation*}
\medskip

{\bf Case 5.} If $t<0$, then $|x'|=\frac12(|x-y|- |x+y|)$ and $|y'|=\frac12(|x-y|+|x+y|)$\,.
By Lemma \ref{lethm1} and Lemma \ref{lem2}(2), a similar argument as Case 4 yields the result.

\medskip

This completes the proof.
\end{proof}

\medskip

\begin{remark}
{\rm
Let
$$f(t,s)\equiv\log\left(1+2\sqrt{\frac{s\sqrt{s^{2}+t^{2}}}{4-s^{2}-t^{2}}}\right),$$
where $t=|x+y|\in [0,2)$ and $s=|x-y|\in [0,2)$\,.

Since $f(t,s)$ is increasing in $t$, then $f(t,s)\geq f(0,s)=\log\left(1+\frac{2s}{\sqrt{4-s^{2}}}\right).$
Hence
\begin{equation*}
\log\left(1+2\sqrt{\frac{|x-y|\sqrt{|x+y|^{2}+|x-y|^{2}}}{4-|x+y|^{2}-|x-y|^{2}}}\right)
\geq\log\left(1+\frac{2|x-y|}{\sqrt{4-|x-y|^{2}}}\right)\,.
\end{equation*}

Moreover, since $|x+y|\left(1+\frac{4}{|x+y|^{2}+|x-y|^{2}}\right)>4$ implies that $t^2+s^2<\frac{4t}{4-t}$\,,
we have
\begin{equation*}
(2-|x+y|)^2+|x-y|^2-(4-|x-y|^2)=2(t^2+s^2)-t^2-4t<\frac{t(t^2-8)}{4-t}\le 0
\end{equation*}
and hence
\begin{equation*}
\log\left(1+\frac{2|x-y|}{\sqrt{(2-|x+y|)^{2}+|x-y|^{2}}}\right)
\geq\log\left(1+\frac{2|x-y|}{\sqrt{4-|x-y|^{2}}}\right)\,.
\end{equation*}

Therefore, the lower estimate of $\tilde{\tau}_{\mathbb{B}^2}$ in Theorem \ref{co1} is better than that in \cite[Theorem 3.2]{MS}.
}
\end{remark}

\bigskip

\begin{nonsec}{\bf The upper half plane case.}
{\rm
We get two formulas for $\tilde{\tau}_{{\mathbb{H}}^{2}}$ in two special cases
in the similar way as in \cite{i1} for calculating the Cassinian metric in the upper half plane. 
For the references of the Cassinian metric, we refer to \cite{hkvz, imsz, kms}. }
\end{nonsec}

\begin{lemma}\label{lem3}
Let $x,y\in\mathbb{H}^{2}$ and $d(x,\partial\mathbb{H}^{2})=d(y,\partial\mathbb{H}^{2})=d$.

(1) If $ |x-y|>2\,d$, then
\begin{equation*}\label{he}
\tilde{\tau}_{{\mathbb{H}}^{2}}(x,y)
=\log\left(1+\sqrt{\frac{|x-y|}{d}}\right).
\end{equation*}

(2) If $ |x-y| \leq 2\, d$, then
\begin{equation*}\label{ge}
\tilde{\tau}_{{\mathbb{H}}^{2}}(x,y)=\log\left(1+\frac{2|x-y|}{\sqrt{4\,d^{2}+|x-y|^{2} }}\right).
\end{equation*}
\end{lemma}
\begin{proof}
Since $\tilde{\tau}_{{\mathbb{H}}^{2}}$ is invariant under translations,
we may assume that $x=(x_{1},x_{2})$ and $y=(-x_{1},x_{2})$,
where $x_{1}=\frac{|x-y|}{2}$ and $x_{2}=d(x,\partial\mathbb{H}^{2})=d(y,\partial\mathbb{H}^{2})=d$.
Let $p=(t,0)$ with $t\ge 0$ and
\begin{align*}
f(t)&=|x-p|^2|y-p|^2 \\
&=t^{4}-2(x_{1}^{2}-x_{2}^{2})t^{2}+(x_{1}^{2}+x_{2}^{2})^2\,.
\end{align*}

\medskip

{\bf Case 1.}
If $ |x-y|>2\,d$, then $x_{1}>x_{2}$ and hence
$$f_{\min}(t)=f({t_{0}})=4x_{1}^{2}x_{2}^{2}\,,$$
where ${t_{0}}=\sqrt{x_{1}^{2}-x_{2}^{2}}$.
Therefore,
\begin{equation*}
\tilde{\tau}_{{\mathbb{H}}^{2}}(x,y)
=\log\left(1+\frac{|x-y|}{\sqrt[4]{f({t_{0}})}}\right)
=\log\left(1+\sqrt{\frac{|x-y|}{d}}\right)\,.
\end{equation*}

\medskip

{\bf Case 2.}
If $ |x-y| \leq 2\, d$, then $x_{1}\leq x_{2}$ and hence
$$f_{\min}(t)=f(0)=(x_{1}^{2}+x_{2}^{2})^{2}\,.$$
Therefore,
\begin{equation*}
\tilde{\tau}_{{\mathbb{H}}^{2}}(x,y)
=\log\left(1+\frac{|x-y|}{\sqrt[4]{f(0)}}\right)
=\log\left(1+\frac{2|x-y|}{\sqrt{4\, d^{2}+|x-y|^{2}}}\right).
\end{equation*}

The proof is complete.
\end{proof}

\medskip

\begin{lemma}\label{lem4}
Let $x,y\in\mathbb{H}^{2}$ with $y-x$ be orthogonal to $\partial\mathbb{H}^{2}$.
Then
\begin{align*}
\tilde{\tau}_{{\mathbb{H}}^{2}}(x,y)=\log\left( 1+\frac{|x-y|}{\sqrt{d(x,\partial \mathbb{H}^{2} )d(y,\partial \mathbb{H}^{2} )}}\right).
\end{align*}
\end{lemma}
\begin{proof}
The proof follows easily from the definition of $\tilde{\tau}$-metric.
\end{proof}

\medskip

\begin{lemma}\label{lethm2}
Let $x,y,x',y'\in\mathbb{H}^2 $ with $x'=\frac{x+y}{2}-\frac{|x-y|}{2}e_{2}$,~$y'=\frac{x+y}{2}+\frac{|x-y|}{2}e_{2}$,
~$x''=\frac{x+y}{2}-\frac{|x-y|}{2}e_{1}$ and $y''=\frac{x+y}{2}+\frac{|x-y|}{2}e_{1}$\,.
Then
\begin{align*}
\tilde{\tau}_{\mathbb{\mathbb{H}}^2}(x'',y'')\leq\tilde{\tau}_{\mathbb{H}^2}(x,y)\leq\tilde{\tau}_{\mathbb{H}^2}(x',y').
\end{align*}
\end{lemma}
\begin{proof}
By symmetry, we may assume that $0<\arg (y-x)<\frac{\pi}{2}$.

Let $b'\in\mathbb{R}$ such that $C(x',y';b')$  is tangent  to $\partial\mathbb{H}^{2}$.
By Proposition \ref{prop2}, there exists a disk $\mathbb{B}^{2}(o',r)\subset \mathbb{H}^2$
such that $C(x',y';b')$ inscribes $\overline{\mathbb{B}^{2}}(o',r)$,
where the center $o'=\frac{x+y}{2}$ and the radius $r=\sqrt{b'^{2}+\frac{|x-y|^{2}}{4}}$.
Moreover, $C(x',y';b') \cap\partial \mathbb{B}^{2}(o',r) \cap\partial \mathbb{H}^{2}$ has one and only one point.

With rotation, $C(x,y;b')\subsetneq\mathbb{H}^{2}$ (see Fig.~\ref{tu10}).
Therefore, there exists a positive number $b\,(> b')$ such that $C(x,y;b)$  is tangent  to $\partial\mathbb{H}^{2}$.
Hence
$$\tilde{\tau}_{\mathbb{H}^{2}}(x,y)\leq\tilde{\tau}_{\mathbb{H}^{2}}(x',y'). $$

With rotation and by Proposition \ref{prop1}, $C(x'',y'';b)\subsetneq\mathbb{H}^{2}$ (see Fig.~\ref{tu11}).
Therefore, there exists a positive number $b''\,(> b)$ such that $C(x'',y'';b'')$  is tangent to $\partial\mathbb{H}^{2}$.
Hence
$$\tilde{\tau}_{\mathbb{H}^2}(x'',y'')\leq\tilde{\tau}_{\mathbb{H}^2}(x,y).$$

This completes the proof.
\end{proof}

\begin{figure}[h]
\begin{minipage}[t]{0.4\linewidth}
\centering
\includegraphics[width=6.92cm]{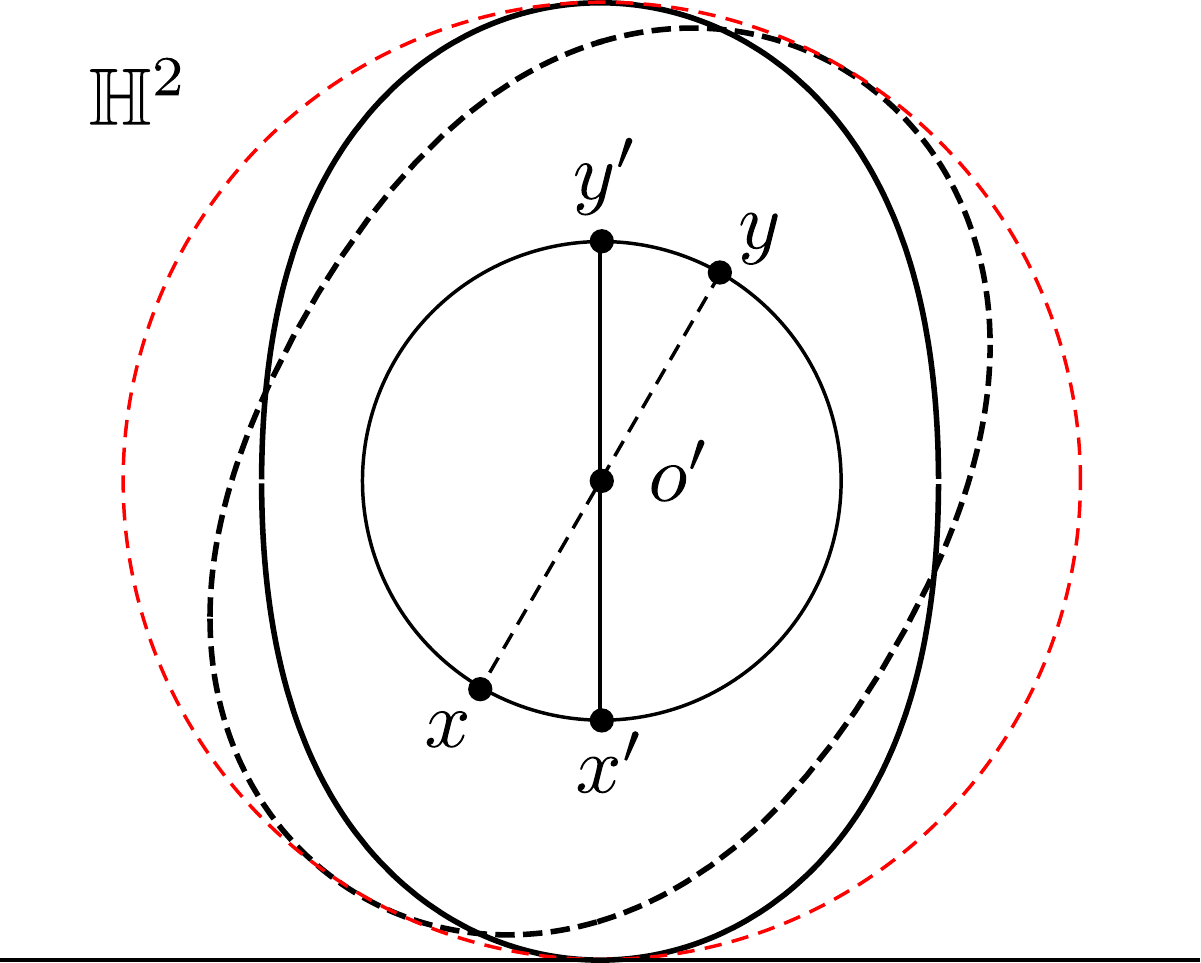}
\caption{\label{tu10} The Cassinian oval $C(x',y';b')$  is tangent  to $\partial\mathbb{H}^{2}$ while $C(x,y;b')\subsetneq\mathbb{H}^{2}$.}
\end{minipage}
\hspace{2.2cm}
\begin{minipage}[t]{0.4\linewidth}
\centering
\includegraphics[width=6.9cm]{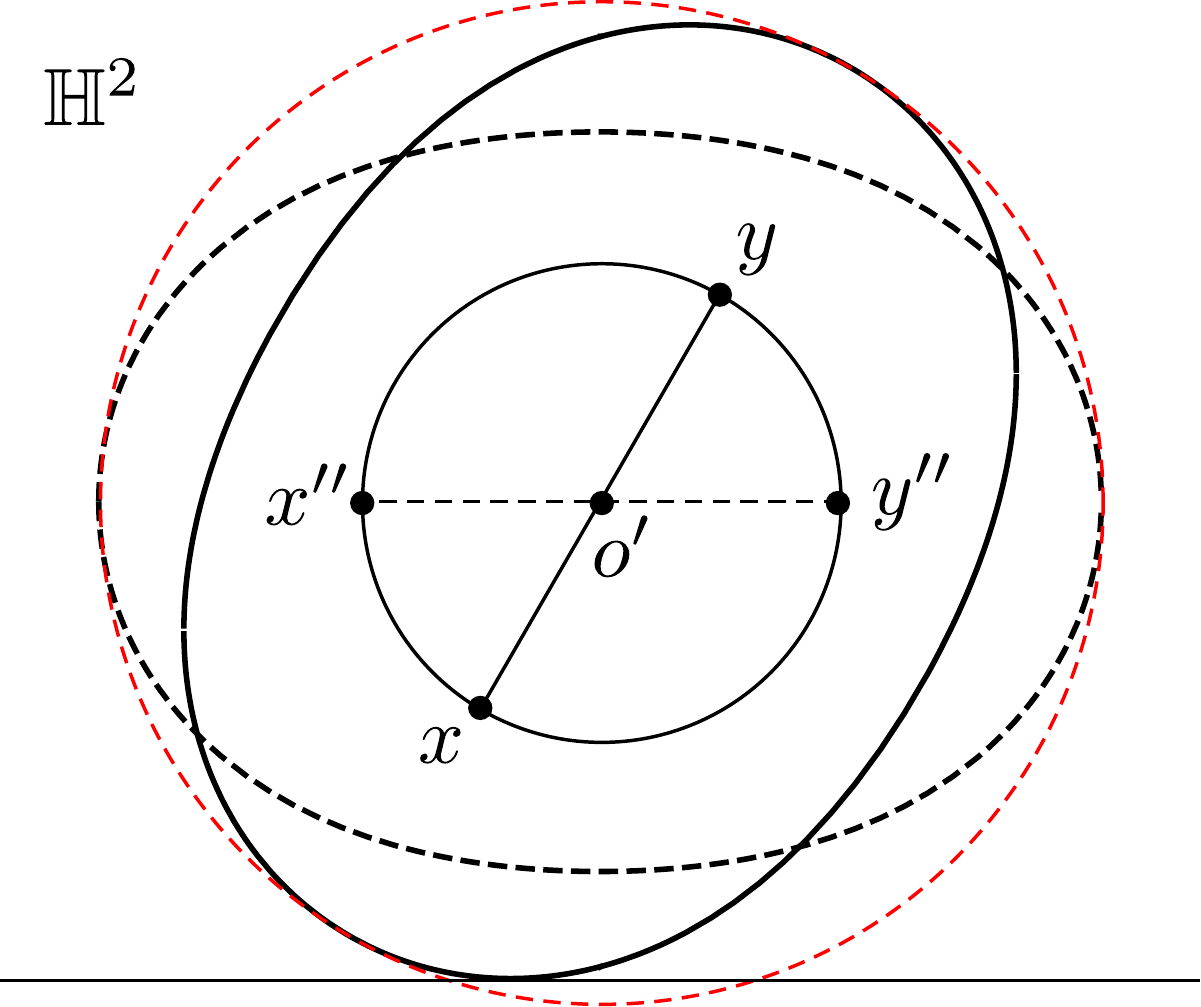}
\caption{\label{tu11} The Cassinian oval $C(x,y;b)$  is tangent  to $\partial\mathbb{H}^{2}$ while $C(x'',y'';b)\subsetneq\mathbb{H}^{2}$. }
\end{minipage}
\end{figure}

\begin{theorem}\label{co3}
Let $x,y\in\mathbb{H}^2 $ and $d=d(\frac{x+y}{2},\partial\mathbb{\mathbb{H}}^{2})$.
Then
\begin{align}\label{thm2ineq2}
\tilde{\tau}_{{\mathbb{H}}^{2}}(x,y)\geq
\left\{
\begin{array}{ll}
\log\left(1+\sqrt{\frac{|x-y|}{d}}\right),
& |x-y| > 2\,d\,,\\
\log\left(1+\frac{2\,|x-y|}{\sqrt{4\,d^{2}+|x-y|^{2}}}\right),
& |x-y| \leq 2\,d,
\end{array}
\right.
\end{align}
and
\begin{equation}\label{thm2ineq3}
\tilde{\tau}_{\mathbb{H}^2}(x,y)\leq \log\left(1+\frac{2|x-y|}{\sqrt{4\, d^{2}- |x-y|^{2} }}\right)\,,
\quad |x-y|<2d\,.
\end{equation}
\end{theorem}
\begin{proof}
To prove inequalities \eqref{thm2ineq2}, let $x''\,,y''$ be the same as in Lemma \ref{lethm2}.
The results follow from Lemma \ref{lethm2} and Lemma \ref{lem3} immediately.

\medskip

To prove inequality \eqref{thm2ineq3}, let $x'\,,y'$ be the same as in Lemma \ref{lethm2}.
Since $d(x',\partial\mathbb{H}^{2})=d-\frac{|x-y|}{2}$ and  $d(y',\partial\mathbb{H}^{2})=d+\frac{|x-y|}{2}$,
together with Lemma \ref{lethm2} and Lemma \ref{lem4}, we have
\begin{align*}
\tilde{\tau}_{\mathbb{H}^2}(x,y)\leq\tilde{\tau}_{\mathbb{H}^2}(x',y')
&=\log\left(1+\frac{|x-y|}{\sqrt{d(x',\partial\mathbb{H}^{2})\,d(y',\partial\mathbb{H}^{2})}}\right)\\
&=\log\left(1+\frac{2|x-y|}{\sqrt{4\,d^2-|x-y|^2}}\right).
\end{align*}

The proof is complete.
\end{proof}

\bigskip

\section{The $\tilde{\tau}$-metric and the hyperbolic metric}\label{tauro}

In \cite{Ibragimov}, Ibragimov showed the relation between the scale invariant Cassinian metric and the hyperbolic metric in the unit ball,
while a statement about the sharpness of comparison was missing.
In this section, we will provide the missing sharpness statement and study the same property in the upper half space.

\begin{theorem}\label{trb}
For all $x,y\in\mathbb{B}^{n}$, we have
\begin{equation}\label{trb1}
\frac{1}{4}\rho_{\mathbb{B}^{n}}(x,y)
\leq\tilde{\tau}_{\mathbb{B}^n}(x,y)
\leq\rho_{\mathbb{B}^{n}}(x,y),
\end{equation}
and both the inequalities are sharp.
In addition, for all $x,y\in\mathbb{B}^{n}$, we have
\begin{equation}\label{trb2}
\tilde{\tau}_{\mathbb{B}^n}(x,y)
\leq\frac{1}{2}\rho_{\mathbb{B}^{n}}(x,y)+\log\frac{5}{4},
\end{equation}
and the inequality is sharp.
\end{theorem}
\begin{proof}
For the inequalities see \cite[Theorem 3.8]{Ibragimov}.

For the sharpness of the left-hand side of inequalities \eqref{trb1},
let $x=-y=t e_{1}$ with $t\in(0,1)$.
By Lemma \ref{lem2}(2) and \eqref{bro},
we have
\begin{equation*}
\lim_{t\rightarrow 1^{-}}\frac{\tilde{\tau}_{\mathbb{B}^n}(x,y)}{\rho_{\mathbb{B}^{n}}(x,y)}
=\lim_{t\rightarrow 1^{-}}\frac{\log\left(1+\frac{2t}{\sqrt{1-t^{2}}}\right)}{2\log\left(1+\frac{2t}{1-t}\right)}
=\frac{1}{2}\lim_{t\rightarrow1^{-}}\frac{\log\left(\frac{2t}{\sqrt{1-t^{2}}}\right)}{\log\left(\frac{2t}{1-t}\right)}
=\frac{1}{4}\,.
\end{equation*}

For the sharpness of the right-hand side of inequalities \eqref{trb1},
let $x=te_{1}$ and $y=(t+(1-t)^{2})e_{1}$ with $t\in(0,1)$.
By Lemma \ref{lem2}(1) and \eqref{bro}, we have
\begin{equation*}
\lim_{t\rightarrow 1^{-}}\frac{\tilde{\tau}_{\mathbb{B}^n}(x,y)}{\rho_{\mathbb{B}^{n}}(x,y)}
=\lim_{t\rightarrow 1^{-}}\frac{\log\left(1+\frac{1-t}{\sqrt{t}}\right)}{\log\left(1+\frac{2(1-t)}{t(1+t)}\right)}
=\lim_{t\rightarrow 1^{-}}\frac{\sqrt{t}(1+t)}{2}
=1\,.
\end{equation*}

\medskip

For the sharpness of  inequality \eqref{trb2}, let $x=(t+(1-t)^{2})e_{1}$ and $y=(t+(1-t)^{5})e_{1}$.
By Lemma \ref{lem2}(1) and \eqref{bro}, we have
\begin{equation*}
\lim_{t\rightarrow 0}\tilde{\tau}_{\mathbb{B}^n}(x,y)
=\lim_{t\rightarrow 0}\log\left(1+\frac{(1-t)(3-3t+t^{2})}
{\sqrt{4-6t+4t^{2}-t^{3}}}\right)
=\log\frac{5}{2}
\end{equation*}
and
\begin{equation*}
\lim_{t\rightarrow 0}\rho_{\mathbb{B}^{n}}(x,y)
=\lim_{t\rightarrow 0}\log\left(\frac{1+t+(1-t)^{2}}{1+t+(1-t)^{5}}(4-6t+4t^{2}-t^{3})\right)
=\log 4\,.
\end{equation*}
Hence
\begin{equation*}
\lim_{t\rightarrow 0}\left(\tilde{\tau}_{\mathbb{B}^n}(x,y)
-\frac{1}{2}\rho_{\mathbb{B}^{n}}(x,y)\right)=\log \frac54\,.
\end{equation*}

This completes the proof.
\end{proof}

\medskip

The following theorem shows the analogue of Theorem \ref{trb} in the upper half space.

\begin{theorem}
For all $x,y\in\mathbb{H}^{n}$, we have
\begin{equation}\label{a}
\frac{1}{4}\rho_{\mathbb{H}^{n}}(x,y)
\leq\tilde{\tau}_{\mathbb{H}^n}(x,y)
\leq\rho_{\mathbb{H}^{n}}(x,y)\,,
\end{equation}
and both the inequalities are sharp.
In addition, for all $x,y\in\mathbb{H}^{n}$, we have
\begin{equation}\label{b}
\tilde{\tau}_{\mathbb{H}^n}(x,y)
\leq\frac{1}{2}\rho_{\mathbb{H}^{n}}(x,y)+\log\frac{5}{4}\,,
\end{equation}
and the inequality is sharp.
\end{theorem}
\begin{proof}
Inequalities \eqref{a} are a consequence of Lemma \ref{le8} and Lemma \ref{le7}.

For the sharpness of the left-hand side of inequalities \eqref{a},
let $x=t e_{1}+e_{n}$ and $y=e_{n}$ with $t>2$.
By Lemma \ref{lem3}(1) and \eqref{cro}, we get
\begin{equation*}
\lim_{t\rightarrow\infty}\frac{\tilde{\tau}_{\mathbb{H}^{n}}(x,y)}{\rho_{\mathbb{H}^{n}}(x,y)}
=\lim_{t\rightarrow\infty}\frac{\log\left(1+\sqrt{t}\right)}{\log\left(1+\frac{t^{2}+\sqrt{t^{4}+4t^{2}}}{2}\right)}
=\frac{1}{4}\,.
\end{equation*}

For the sharpness of the right-hand side of inequalities \eqref{a},
let $x=t e_{n}$ and $y=\frac{1}{t}e_{n}$ with $t>1$.
By Lemma \ref{lem4} and \eqref{hro}, we get
\begin{equation*}
\lim_{t\rightarrow1^+}\frac{\tilde{\tau}_{\mathbb{H}^{n}}(x,y)}{\rho_{\mathbb{H}^{n}}(x,y)}
=\lim_{t\rightarrow1^+}\frac{\log(1+(t-\frac{1}{t}))}{\log(1+(t^{2}-1))}
=1\,.
\end{equation*}

To prove inequality \eqref{b}, we first observe that
\begin{equation*}
\inf_{\xi\in\mathbb{H}^{n}}|x-\xi||y-\xi|\geq x_{n}y_{n}\,,\quad {\rm for\,\, all}\,\, x,y\in\mathbb{H}^{n}\,,
\end{equation*}
and the equality holds when $y-x$ is orthogonal to $\partial\mathbb{H}^{n}.$

Since
\begin{equation*}
1+\sqrt{2({\rm ch}\,t-1)}=1+e^{t/2}-e^{-t/2}\leq\frac{5}{4}e^{t/2}\,,\quad {\rm for\,\, all}\,\, t\geq0\,,
\end{equation*}
together with \eqref{cro}, we have
\begin{align*}
\tilde{\tau}_{\mathbb{H}^{n}}(x,y)
&\leq\log\left(1+\frac{|x-y|}{\sqrt{x_{n}\,y_{n}}}\right)\\
&=\log\left(1+\sqrt{2({\rm ch}\,\rho_{\mathbb{H}^{n}}(x,y)-1)}\right)\\
&\le\frac{1}{2}\,\rho_{\mathbb{H}^{n}}(x,y)+\log\frac{5}{4}\,.
\end{align*}

To prove the sharpness of inequality \eqref{b},
let $x=2 e_{n}$ and $y=\frac{1}{2}e_{n}$.
By Lemma \ref{lem4} and \eqref{hro},
we get
\begin{equation*}
\tilde{\tau}_{\mathbb{H}^{n}}(x,y)=\log\frac{5}{2}
\quad {\rm and} \quad
\rho_{\mathbb{H}^{n}}(x,y)=\log 4\,.
\end{equation*}
Hence
\begin{equation*}
\tilde{\tau}_{\mathbb{H}^n}(x,y)
=\frac{1}{2}\rho_{\mathbb{H}^{n}}(x,y)+\log\frac{5}{4}.
\end{equation*}

This completes the proof.
\end{proof}

\bigskip

\section{The $\tilde{\tau}$-metric and M\"obius transformations}\label{taumob}

Ibragimov studied the distortion properties of $\tilde{\tau}$-metric under M\"obius transformations of the unit ball \cite[Theorem 4.1, Theorem 4.2]{Ibragimov}. Later,
Mohapatra and Sahoo considered the same problem in the punctured unit ball \cite[Theorem 2.1]{MS}.
For instance, the following quasi-isometry property was obtained in \cite{Ibragimov}.

\begin{theorem}\cite[Theorem 4.2]{Ibragimov}
If $f: \overline{{\mathbb R}} ^n\to \overline{{\mathbb R}}^n$ is a M\"{o}bius transformation with $f\mathbb{B}^n=\mathbb{B}^n$,
then for all $x,y\in \mathbb{B}^n$, we have
\begin{equation}\label{ibragimov-quasiiso}
\frac{1}{2}\tilde{\tau}_{\mathbb{B}^n}(x,y)-\log \frac 54\leq\tilde{\tau}_{\mathbb{B}^n}(f(x),f(y))\leq 2\tilde{\tau}_{\mathbb{B}^n}(x,y)+\log \frac 54\,.
\end{equation}
\end{theorem}

In this section, we continue the investigation on the distortion of $\tilde{\tau}$-metric of general domains under M\"obius transformations. In particular, we show that in \eqref{ibragimov-quasiiso} the additional constant $\log\frac{5}{4}$ can be removed and the constants $\frac{1}{2}$ and $2$ are the best possible.

\begin{theorem}\label{moblip}
If $D$ and $D'$ are proper subdomains of $\mathbb{R}^{n}$ and
if $f: \overline{{\mathbb R}} ^n\to \overline{{\mathbb R}}^n$ is a M\"{o}bius transformation with $fD=D'$,
then for all $x,y\in D$, we have
\begin{equation*}
\frac{1}{2}\tilde{\tau}_{D}(x,y)\leq\tilde{\tau}_{D'}(f(x),f(y))\leq 2\tilde{\tau}_{D}(x,y)\,.
\end{equation*}
\end{theorem}
\begin{proof}
The proof follows from Lemma \ref{lelip2} and Lemma \ref{lelip1} immediately.
\end{proof}

\medskip

The following theorem shows that the above constants $\frac{1}{2}$ and $2$  can not be improved.

\begin{theorem}\label{thlip}
Let $f: {\mathbb H}^n \to {\mathbb B}^n=f{\mathbb H}^n$  be a M\"obius transformation.
Then for all $x,y\in\mathbb{H}^{n}$\,, we have
\begin{equation*}
\frac{1}{2}\tilde{\tau}_{\mathbb{H}^{n}}(x,y)\leq\tilde{\tau}_{\mathbb{B}^{n}}(f(x),f(y))\leq 2\tilde{\tau}_{\mathbb{H}^{n}}(x,y),
\end{equation*}
and the constants $\frac{1}{2}$ and $2$ are the best possible.
\end{theorem}
\begin{proof}
The inequalities are clear from Theorem \ref{moblip}. Next we show the sharpness of the inequalities.

Since $\tilde{\tau}$-metric is invariant under translations and stretchings of $\mathbb{H}^{n}$ onto itself
and orthogonal transformations of $\mathbb{B}^{n}$ onto itself,
it suffices to consider
\begin{equation*}
f(z)=-e_n+\frac{2(z+e_n)}{|z+e_n|^2}\,.
\end{equation*}

Let $x=t e_n$ and $y=\frac{1}{t} e_n$ with $t>1\,.$
Then
\begin{equation*}
f(x)=-\frac{t-1}{t+1}e_n \quad\quad\quad~~~ {\rm and} \quad\quad\quad~~~  f(y)=\frac{t-1}{t+1}e_n\,.
\end{equation*}

By Lemma \ref{lem2}(2) and Lemma \ref{lem4},  we get
\begin{equation*}
\lim_{t\rightarrow\infty}\frac{\tilde{\tau}_{\mathbb{B}^{n}}\left(f(x),f(y)\right)}{\tilde{\tau}_{\mathbb{H}^{n}}(x,y)}
=\lim_{t\rightarrow\infty}\frac{\log\left(1+(\sqrt t-\frac{1}{\sqrt{t}})\right)}{\log\left(1+(t-\frac{1}{t})\right)}
=\lim_{t\rightarrow\infty}\frac{\log\left(\sqrt t-\frac{1}{\sqrt{t}}\right)}{\log\left(t-\frac{1}{t}\right)}
=\frac{1}{2}\,.
\end{equation*}

\medskip

Let $x=t e_{1}+e_n$ and $y=e_n$ with $t>2$.
Then
\begin{equation*}
f(x)=\frac{2t}{t^{2}+4} e_{1}-\frac{t^2}{t^2+4} e_n \quad\quad\quad~~~ {\rm and} \quad\quad\quad~~~  f(y)=0\,.
\end{equation*}

By Remark \ref{rmk-tau0x} and Lemma \ref{lem3}(1),  we get
\begin{equation*}
\lim_{t\rightarrow\infty}\frac{\tilde{\tau}_{\mathbb{B}^{n}}\left(f(x),f(y)\right)}{\tilde{\tau}_{\mathbb{H}^{n}}(x,y)}
=\lim_{t\rightarrow\infty}\frac{\log\left(1+\frac{t}{\sqrt{t^{2}+4-t\sqrt{t^{2}+4}}}\right)}{\log\left(1+\sqrt{t}\right)}\\
=\lim_{t\rightarrow\infty}\frac{\log\left(\frac{t}{\sqrt{t^2+4-t\sqrt{t^2+4}}}\right)}{\log{\sqrt{t}}}
=2\,.
\end{equation*}

This completes the proof.
\end{proof}

\subsection*{Acknowledgments}
This research was supported by National Natural Science Foundation of China (NNSFC) under Grant No.11771400 and No.11601485 ,
and Science Foundation of Zhejiang Sci-Tech University (ZSTU) under Grant No.16062023\,-Y.



\end{document}